\documentclass[11pt]{article}
\usepackage{mathrsfs}
\usepackage{amsfonts,amsmath,dsfont,amssymb,amsthm,stmaryrd,bbm,color}
\usepackage[hidelinks]{hyperref}

\newtheorem{conj}{Conjecture}[section]
\newtheorem{thm}{Theorem}[section]
\newtheorem{rmk}[conj]{Remark}
\newtheorem{lem}[conj]{Lemma}
\newtheorem{prop}[conj]{Proposition}

\newtheorem{defn}[conj]{Definition}
\newtheorem{coro}[conj]{Corollary}

\newcommand{\supp}{\mathrm{supp}}

\newcommand{\R}{\mathbb{R}} 

\newcommand{\Z}{\mathbb{Z}}

\newcommand{\E}{\mathbb{E}}

\def\P{{\mathbb P}}

\begin{document}
\title{Boolean functions: noise stability, non-interactive correlation distillation, and mutual information}
\author{Jiange Li, Muriel M\'edard
\thanks{
These results were presented in part at the 2018 International Symposium on Information Theory, Colorado, USA. J. Li was with the Research Laboratory of Electronics (RLE) at the Massachusetts Institute of Technology (MIT). He is now with the Einstein Institute of Mathematics at the Hebrew University of Jerusalem. M. M\'edard is with the Research Laboratory of Electronics (RLE) at the Massachusetts Institute of Technology (MIT). Email: {\tt jiange.li@mail.huji.ac.il, medard@mit.edu}
}}
\date{\today}
\maketitle


\begin{abstract}
Let $T_{\epsilon}$ be the noise operator acting on Boolean functions $f:\{0, 1\}^n\to \{0, 1\}$, where $\epsilon\in[0, 1/2]$ is the noise parameter. Given $\alpha>1$ and fixed mean $\E f$, which Boolean function $f$ has the largest $\alpha$-th moment $\E(T_\epsilon f)^\alpha$? This question has close connections with noise stability of Boolean functions, the problem of non-interactive correlation distillation, and Courtade-Kumar's conjecture on the most informative Boolean function. In this paper, we characterize maximizers in some extremal settings, such as low noise ($\epsilon=\epsilon(n)$ close to 0), high noise ($\epsilon=\epsilon(n)$ close to 1/2), as well as when $\alpha=\alpha(n)$ is large. Analogous results are also established in more general contexts, such as Boolean functions defined on discrete torus $(\Z/p\Z)^n$ and the problem of noise stability in a tree model. 
\end{abstract}


\section{Introduction}

Let $\epsilon\in[0, 1/2]$ be the noise parameter. Let $T_{\epsilon}$ be the noise operator (Definition \ref{defn:noise}) acting on Boolean functions $f:\{0, 1\}^n\to \{0, 1\}$ on the discrete cube associated with the uniform measure. In this paper, we investigate the problem that given $\alpha> 1$ and fixed mean $\E f$ which Boolean function $f$ has the largest $\alpha$-th moment $\E(T_\epsilon f)^\alpha$. This question has close connections with other problems in information theory and theoretical computer science, such as noise stability of Boolean functions, non-interactive correlation distillation (NICD), as well as the conjecture of Courtade-Kumar on the most informative Boolean function.

\textbf{Noise stability}. The second moment $\E(T_\epsilon f)^2$ is know as nose stability of $f$, in which case $f$ could be an arbitrary real-valued function. We may refer to $\E(T_\epsilon f)^\alpha$ as \emph{$\alpha$-stability}. Suppose that $f$ is a Boolean function with the support $S\subseteq\{0, 1\}^n$. Then we have the following probabilistic interpretation 
\begin{equation}\label{eq:2nd moment}
\E(T_\epsilon f)^2=\P(X\in S, Y\in S).
\end{equation}
Here, $X$ and $Y$ are uniformly random binary sequences with the correlation matrix $\rho I$, where $\rho=(1-2\epsilon)^2$ and $I$ is the identity matrix. This isoperimetric problem has been  extensively studied in the Gaussian setting, in which case $X$ and $Y$ are standard Gaussian vectors with correlation matrix $\rho I$. By the Gaussian isoperimetric inequality \cite{ST74, Bor75}, the probability in \eqref{eq:2nd moment} is maximized when $S$ is a half space, and this was generalized by Isaksson and Mossel \cite{IM12} to more than two Gaussian vectors. The optimality of half spaces was further extended to other functionals, such as $\E(T_\epsilon f)^\alpha$ for $\alpha>1$ by  Eldan \cite{Eld15}, and $\E \Phi(T_\epsilon f)$ for convex functions $\Phi$ by  Kindler, O'Donnell and Witmer \cite{KODW15}. Extremizers of  $\E(T_\epsilon f)^\alpha$ in the discrete setting are yet to be understood.

\textbf{Non-interactive correlation distillation (NICD)}. Our second motivation comes from the problem of NICD in information theory and computer science \cite{MOD05, MODRSS06, Yang07}, which is relevant to cryptographic information reconciliation, random beacons in cryptography and security, and coding theory. 
In its most basic form, the problem of NICD involves two players. Let $X$ be a uniformly random binary string transmitted to Alice and Bob through  independent binary symmetric channels with cross-over probability $\epsilon$  ($\text{BSC}(\epsilon)$).  
Upon receiving $Y$ and $Y'$, Alice and Bob output one bit without communication. Their goal is to maximize the probability that they give the same bit, i.e., $\P(f(Y)=g(Y'))$, where $f$ and $g$ are Boolean functions. Notice that 
\begin{align*}
\E f(Y) &=\P(f(Y)=g(Y')=1)+\P(f(Y)=1, g(Y')=0),\\
\E g(Y') &=\P(f(Y)=g(Y')=1)+\P(f(Y)=0, g(Y')=1).
\end{align*}
We have
$$
\P(f(Y)=g(Y'))=1+2\E f(Y)g(Y')-\E f(Y)-\E g(Y').
$$
Therefore, given $\E f$ and $\E g$, it suffices to maximize $\E f(Y)g(Y')=\E (T_{\epsilon}fT_{\epsilon}g)$.  Similarly, the goal of the $k$-player NICD problem is to maximize $\P(f_1(Y^1)=\cdots=f_k(Y^k))$, where $Y^1, \cdots, Y^k$ are $k$ noise corrupted versions of $X$, and $f_1, \cdots, f_k$ are Boolean functions. In general, this is not equivalent to the maximal correlation $\E\prod_{i=1}^kf_i(Y^i)=\E\prod_{i=1}^kT_\epsilon f_i$. If one is happy with the restriction $f_1=\cdots=f_k=f$, then we have
$$
\P(f(Y^1)=\cdots=f(Y^k))=\E((T_\epsilon f)^k+(1-T_\epsilon f)^k).
$$
In this case, the $k$-player NICD problem can be rephrased as the maximization of $\E\Psi_k(T_\epsilon f)$, where $\Psi_k(x)=x^k+(1-x)^k$ for $x\in[0, 1]$.
(This was also observed in \cite{MOD05}, Remark 1.2). Maximizers of $\E(T_\epsilon f)^k$ often possess certain special structures, which enable that $\E(T_\epsilon f)^k$ and $\E(T_\epsilon(1- f))^k$ often achieve their maximums at the same Boolean function. Hence, it often suffices to study the maximization of $\E(T_\epsilon f)^\alpha$.

\textbf{Courtade-Kumar's conjecture}. Let $X\in\{0, 1\}^n$ be a uniformly random binary sequence. Let $Y$ be the output of $X$ through a $\text{BSC}(\epsilon)$ channel. Which Boolean function $f$ maximizes the mutual information $I(X; f(Y))$ between $X$ and $f(Y)$? This is known as the most informative Boolean function problem raised by Courtade and Kumar \cite{CK14}. They also conjectured that the mutual information $I(X; f(Y))$  is maximized by the dictator function. The analogous question in the Gaussian setting was verified by Kindler, O'Donnell and Witmer \cite{KODW15}. Pichler, Piantanida and Matz \cite{PPM16} proved the variant that the dictator function maximizes the mutual information $I(f(X); g(Y))$ among all Boolean functions $f$ and $g$. The original conjecture is still wide open. Courtade and Kumar \cite{CK14} has observed that their conjecture holds in extremal scenarios $\epsilon=\epsilon(n)\to 0, 1/2$. Quantitative bounds  can be found in \cite{OSW15}. Samorodnitsky \cite{Sam16} proved Courtade-Kumar's conjecture in the high noise setting, i.e., for $\epsilon\in[\epsilon_0, 1/2]$, where $0<\epsilon_0<1/2$ is some absolute constant. We observe that Courtade-Kumar's conjecture is closely related to the $\alpha$-NICD problem, where $\alpha>1$ is not necessarily an integer. More precisely, for fixed $\E f$, if there is a unique (up to isometric equivalence) maximizer $f$ of $\E\Psi_\alpha(T_\epsilon f)$ for $\alpha\in(1, \alpha_0)$, where $\Psi_\alpha(x)=x^\alpha+(1-x)^\alpha$ for $x\in[0, 1]$ and $\alpha_0>1$ can be dimension dependent, then $f$ also maximizes $I(X; f(Y))$ among all Boolean functions with the same expectation $\E f$. Conversely, for fixed $\E f$, if $f$ is the unique (up to isometric equivalence) maximizer  of $I(X; f(Y))$, then it also maximizes $\E\Psi_\alpha(T_\epsilon f)$ for $\alpha\in (1, \alpha_0)$ among all Boolean functions with the same expectation $\E f$. In particular, for balanced Boolean functions, Courtade-Kumar's conjecture holds if the dictator function is the unique maximizer of $\E(T_\epsilon f)^\alpha$ for $\alpha\in(1, \alpha_0)$. This is another motivation for us to study the maximization of  $\E(T_\epsilon f)^\alpha$.

The paper is organized as follows. In Section \ref{sec:noise}, we give a brief account of noise operator and total influence of Boolean functions, and we refer the interested reader to the monograph \cite{OD14} for further information. In Section \ref{sec:main}, we include results in asymptotic settings, such as  low noise ($\epsilon=\epsilon(n)$ close to 0), high noise ($\epsilon=\epsilon(n)$ close to 1/2), as well as when $\alpha=\alpha(n)$ is large. In Section \ref{sec:mutual}, we relate the $\alpha$-NICD problem to Courtade-Kumar's conjecture on the most informative Boolean function. In Section \ref{sec:extensions}, we establish analogous results in more general contexts, such as Boolean functions defined on the discrete torus $(\Z/p\Z)^n$ and the problem of noise stability in a tree model. We conclude the paper with a brief discussion of potential applications and future work in Section \ref{sec:conclusion}.


\section{Noise operator and total influence}\label{sec:noise}

We associate the discrete cube $\{0, 1\}^n$ with the uniform measure $\mu$. The set of functions $W_A(x)=(-1)^{\sum_{i\in A}x_i}, A\subseteq[n]$ forms an orthonormal basis; that is, $\E(W_A)^2=1$ and $\E(W_AW_B)=0$ for $A\neq B$. (The expectation is taken with respect to the reference measure $\mu$. We always omit this when it is clear from the context). 
Any real-valued function $f$ on $\{0, 1\}^n$ has the following Fourier expansion
\begin{equation}\label{eq:fourier-expan}
f(x)=\sum_{A\subseteq[n]}\hat{f}(A)W_A(x),
\end{equation}
where $\hat{f}(A)=\E(fW_A)$ are Fourier coefficients. In particular, one has $\hat{f}(\emptyset)=\E f$. 

\begin{defn}\label{defn:noise}
Let $0\leq \epsilon\leq 1/2$. The noise operator $T_\epsilon$ acts on $f:\{0, 1\}^n\to \R$  as follows
\begin{equation}\label{eq:prob-int}
T_\epsilon f(x)=\E f(x+Z),
\end{equation}
where $Z$ has independent Bernoulli($\epsilon$) coordinates and the addition is modular by 2.
\end{defn}

One can see that $T_\epsilon$ is a convolution operator. To be more precise, we have $T_\epsilon f=\mu_\epsilon\ast f$, where $\mu_\epsilon$ is the distribution of $Z$.  One can also think of the operator $T_\epsilon$ as follows. Let $X\in\{0, 1\}^n$ be a binary sequence selected uniformly at random. Let $Y$ be the output of $X$ through a $\text{BSC}(\epsilon)$ channel.
Then, we have $T_\epsilon f(x)=\E[f(Y)|X=x]$; that is, $T_\epsilon f$ is the average of $f$ over the outputs. As $\epsilon$ grows, the channel becomes more noisy, and the output is more random, and $T_\epsilon f$ becomes more ``regular". In particular, we have $T_0 f=f$ and $T_{1/2}f=\E f$. This regularizing effect can also been seen from the following Fourier expansion
\begin{equation}\label{eq:noise-fourier}
T_\epsilon f(x)=\sum_{A\subseteq[n]}(1-2\epsilon)^{|A|}\hat{f}(A)W_A(x).
\end{equation}

For $1\leq i\leq k$, we define $Y^i=X+Z^i$, where $Z^i$ are independent copies of $Z$. One can check that each pair $(Y^i, Y^j)$ for $i\neq j$ has the correlation matrix $\rho I$, where $\rho=(1-2\epsilon)^2$ and $I$ is the identity matrix. For simplicity, we say that they are {\emph {$\rho$-correlated}}.   Since $Y_i$ are independent given $X$, together with \eqref{eq:prob-int}, the conditioning argument yields
$$
\E\prod_{i=1}^kf_i(Y^i)=\E\prod_{i=1}^kT_\epsilon f_i.
$$
Owing to this relation, our results below will be stated in terms of either LHS or RHS of the above identity. 

The noise operator introduced before can be thought of as a special type of Markov semi-groups of Markov chains on graphs. (In our case, the underlying graph is the discrete cube). Hence, it may be worth to investigate the problem of $\alpha$-stability in more general contexts. To be more precise, let us consider the following simple continuous time Markov chain on a simple connected undirected graph $G=(V, E)$. Each vertex $x\in V$ is associated with an exponential clock, i.e., an exponential random variable with parameter 1. When the clock rings, the chain jumps from the current vertex to its neighbours with equal probability. The transition matrix of this Markov chain is $K=D^{-1}A$, where $A$ is the adjacency matrix and $D$ is the diagonal matrix with $D(x, x)=\text{deg}(x)$ the degree of $x$. The invariant measure of the Markov chain is $\mu(x)=\text{deg}(x)/\sum_{y\in V}\text{deg}(y)$. 
The Markov semi-group $(P_t)_{t\geq0}$ acts on $f: V\to \R$ as follows
\begin{equation}\label{eq:semi-group}
P_tf(x)=e^{-tL}f(x),
\end{equation}
where $L=I-K$ is the Laplacian. For any function $f$, we have $\E(Lf)=0$, where the expectation is taken with respect to the invariant measure of the Markov chain. This follows from differentiating the equation $\E(P_tf)=\E f$ with respect to $t$ at $t=0$. By Jensen's inequality, we have $P_t(\Phi(f))\geq\Phi(P_t f)$ for convex functions $\Phi$. Differentiating this inequality  with respect to $t$ at $t=0$, we have $L(\Phi(f))\geq \Phi'(f)L f$. Therefore, we have
$$
\frac{d}{dt}\E \Phi(P_t f)=\E\Phi'(P_tf)L(P_tf)\leq \E L(\Phi(P_tf))=0,
$$
i.e., $\E \Phi(P_t f)$ is a decreasing function of $t$. We refer the interested reader to the monograph \cite{BGL14} for more general theory of Markov semi-groups. 

An important notation used in the study of Boolean functions is influence. We first introduce the flipping operator $\sigma_i$ defined as follows
\begin{equation}\label{eq:flipping}
\sigma_i(x_1,\cdots, x_i, \cdots, x_n)=(x_1,\cdots, 1-x_i, \cdots, x_n),
\end{equation}
i.e., $\sigma_i$ only changes the value of the $i$-th coordinate.
\begin{defn}\label{defn:influence}
Let $f:\{0, 1\}^n\to\{0, 1\}$ be a Boolean function. The influence of the $i$-th variable $I_i(f)$ is defined as
$$
I_i(f)=\mu(x: f(x)\neq f(\sigma_i(x))).
$$
The total influence $I(f)$ is defined as 
$$
I(f)=\sum_{i=1}^nI_i(f).
$$ 
\end{defn}

We have the following geometrical interpretation of influence in terms of edge boundary. Let $S$ be the support of $f$. The $i$-th direction edge boundary $\partial_iS$ is defined as
$$
\partial_i S=\{(x, \sigma_i(x)): x\in S, \sigma_i(x)\notin S\}.
$$
Two vertices $x, y\in \{0, 1\}^n$ are called adjacent, i.e., $x\sim y$, if and only if their Hamming distance is 1. The edge boundary $\partial S$ is defined as 
$$
\partial S=\{(x, y): x\sim y, x\in S, y\notin S\}.
$$ 
It is easy to see that $\partial S=\cup_{i=1}^n\partial_iS$.  One can check the following identities
\begin{align}
I_i(f) &=\frac{|\partial_i S|}{2^{n-1}},\label{eq:geo-i-influence} \\
I(f) &=\frac{|\partial S|}{2^{n-1}}. \label{eq:geo-influence}
\end{align}

We also have the following Fourier analytic representation of influence. Since $f$ takes values 0 or 1, one can rewrite $I_i(f)$ as
$$
I_i(f)=\E(f(X)-f(\sigma_i(X)))^2,
$$
where $X\in\{0, 1\}^n$ is a uniformly random binary string. Using the Fourier expansion \eqref{eq:fourier-expan}, we have
$$
I_i(f)=\E\Big(2\sum_{A\ni i}\hat{f}(A)W_A(x)\Big)^2=4\sum_{A\ni i}\hat{f}(A)^2,
$$
and
$$
I(f)=4\sum_{A\subseteq[n]}|A|\hat{f}(A)^2.
$$


\section{Main results}\label{sec:main}

For the problem of $k$-player correlation, the following statement asserts that the players should use the same strategy to maximize their correlation. 

\begin{prop}\label{prop:holder}
Let $0<\rho<1$. Let $Y^1, \cdots, Y^k\in\{0, 1\}^n$ be $\rho$-correlated uniformly random binary strings. For any functions $f_i:\{0, 1\}^n\to [0, \infty)$, we have
$$
\E \prod_{i=1}^k f_i(Y^{i})\leq\max_{1\leq i\leq k}\E\prod_{j=1}^kf_i(Y^j).
$$
Equality is achieved if and only if $f_i$ are identical.
\end{prop}

\begin{proof}
As shown before, we can realize $Y^i$ as $X+Z^i$, where $X$ is a uniformly random binary string, and the coordinates of $Z^i$ are i.i.d. Bernoulli($\epsilon$) with $\epsilon=(1-\sqrt{\rho})/2$. Since $Y^i$ are independent given $X$, we have
\begin{align*}
\E\prod_{i=1}^k f_i(Y^{i}) 
&=\E \prod_{i=1}^k \E[f_i(Y^{i})|X]\\
&\leq\prod_{i=1}^k (\E(\E[f_i(Y^{i})|X])^k)^{1/k}\\
&=\prod_{i=1}^k\Big(\E\prod_{j=1}^k\E[f_i(Y^j)|X]\Big)^{1/k}\\
&=\prod_{i=1}^k\Big(\E\prod_{j=1}^kf_i(Y^j)\Big)^{1/k}\\
&\leq\max_{1\leq i\leq k}\E\prod_{j=1}^kf_i(Y^j).
\end{align*}
The first inequality follows from H\"{o}lder's inequality and equality is achieved if and only if $\E[f_i(Y^{i})|X]$ are multiples of the same function. Since the noise operator is invertible, $f_i$ are also multiples of the same function. The equality case in the second inequality requires that these multiples are the same. This concludes the proof.
\end{proof}

We call two subsets $A, B\subseteq \{0, 1\}^n$ {\emph {isometrically equivalent}} if $B=\phi(A+a)$ for some permutation $\phi$ on $[n]$ and some $a\in\{0, 1\}^n$. Here, the subset $\phi(A)$ is obtained from $A$ by applying $\phi$ to coordinates of all vectors in $A$. Two Boolean functions are called isometrically equivalent if their supports are isometrically equivalent. (This isometric equivalence was also defined in \cite{MOD05}, $\pi_S$ borrowing the notation therein, although they did not call it in this way). Let $f$ and $g$ be isometrically equivalent Boolean functions with supports $A$ and $B=\phi(A+a)$, respectively. One can check that $T_\epsilon g(x)=T_\epsilon f(\phi^{-1}(x)+a)$, where $\phi^{-1}$ is the inverse permutation of $\phi$. This implies that the functional $\E(T_\epsilon f)^\alpha$ is isometrically invariant. Hence, our results on extremal Boolean functions throughout the paper always hold up to isometric equivalence.

A subset $A\subseteq\{0, 1\}^n$ is called a {\emph{lexicographic}} set if it is the initial segment of $\{0, 1\}^n$ labelled in the lexicographic ordering. For example, $A=\{(0, 0, 0), (0, 0, 1), (0, 1, 0), (0, 1, 1)\}$ is the lexicographic set of $\{0, 1\}^3$ with 4 elements. We call a Boolean function lexicographic if its support is a lexicographic set. A well-known result of Harper \cite{Har64} asserts that the sets with the least edge boundary among all subsets of $\{0, 1\}^n$ of fixed size are isometrically equivalent to the lexicographic set. Owing to the connection between total influence and edge boundary, Harper's theorem is equivalent to that, up to isometric equivalence, the lexicographic function is the unique minimizer of total influence among all Boolean functions with fixed mean. 

\begin{thm}\label{thm:high-low}
Let $\alpha>1$. Let $\E f$ be fixed. When $\epsilon=\epsilon(n)$ is sufficiently small, the quantity $\E(T_\epsilon f)^\alpha$ is maximized by the lexicographic function. When $\epsilon=\epsilon(n)$ is sufficiently close to $1/2$,  the quantity $\E(T_\epsilon f)^\alpha$ is maximized by some Boolean function with the largest degree-1 Fourier weight $W_1(f)=\sum_{i=1}^n\hat{f}(\{i\})^2$. Moreover, if $f$ is assumed to be balanced, i.e., $\P(f=0)=\P(f=1)$, the dictator function $f(x)=x_1$ maximizes $\E(T_\epsilon f)^\alpha$ in both scenarios. 
\end{thm} 
\begin{proof}
As mentioned in the paragraph below equation \eqref{eq:prob-int}, one can interpret $T_\epsilon$ as taking conditional expectation. Then, Jensen's inequality implies that for any convex function $\Phi$, the functional $\E\Phi(T_\epsilon f)$ is decreasing with respect to $\epsilon\in[0, 1/2]$. In particular, $\E(T_\epsilon f)^\alpha$ is a decreasing function of $\epsilon$ when $\alpha>1$.  Notice that the initial value $\E(T_{0} f)^\alpha=\E f$ is fixed. Hence, $\E(T_\epsilon f)^\alpha$ with smaller decreasing rates will have larger values.
Differentiating the function $\E(T_\epsilon f)^\alpha$ with respect to $\epsilon$, we have
\begin{equation}\label{eq:derivative}
\frac{d}{d\epsilon}\E(T_\epsilon f)^\alpha=2\alpha(1-2\epsilon)^{-1}\E(T_\epsilon f)^{\alpha-1}(L\circ T_\epsilon) f,
\end{equation}
where the operator $L$ is defined as
\begin{equation}\label{eq:infi-opr}
Lf(x)=-\sum_{A\subseteq[n]}|A|\hat{f}(A)W_A(x).
\end{equation} 
In particular, we have
\begin{align}\label{eq:der-0}
\frac{d}{d\epsilon}\E(T_\epsilon f)^\alpha\Big|_{\epsilon=0} &= 2\alpha\E (f^{\alpha-1}Lf)= 2\alpha\E (fLf)\nonumber\\
&=-2\alpha\sum_{A\subseteq [n]}|A|\hat{f}(A)^2 \notag\\
&=-\frac{\alpha}{2}\cdot I(f),
\end{align}
where $I(f)$ is the total influence of $f$ (Definition \ref{defn:influence}). Harper's theorem \cite{Har64} and \eqref{eq:der-0} imply that  the lexicographic function is the unique maximizer of $\frac{d}{d\epsilon}\E(T_\epsilon f)^\alpha\big|_{\epsilon=0}$. Hence, the lexicographic function maximizes $\E(T_\epsilon f)^\alpha$ when $\epsilon=\epsilon(n)$ is sufficiently small.

When $\epsilon=\epsilon(n)$ is sufficiently close to $1/2$, we can prove the statement in a similar manner. Notice that $\E(T_{1/2} f)^\alpha=(\E f)^{\alpha}$ is fixed as the ending value. Again, using the fact that $\E(T_\epsilon f)^\alpha$ is decreasing, functions $\E(T_\epsilon f)^\alpha$ decreasing faster  will have larger values around $\epsilon=1/2$. In this case, we have
$$
T_\epsilon f(x)=\E f+(1-2\epsilon)\sum_{i=1}^n\hat{f}(\{i\})(-1)^{x_i}+O((1-2\epsilon)^2).
$$
Then we have
\begin{align}
(L\circ T_\epsilon) f(x) &=-(1-2\epsilon)\sum_{i=1}^n\hat{f}(\{i\})(-1)^{x_i}+O((1-2\epsilon)^2), \label{eq:1}\\
(T_\epsilon f(x))^{\alpha-1} &=(\alpha-1)(\E f)^{\alpha-2}(1-2\epsilon)\sum_{i=1}^{n}\hat{f}(\{i\})(-1)^{x_i}+(\E f)^{\alpha-1}+O((1-2\epsilon)^2). \label{eq:2}
\end{align}
The terms $O(\cdot)$ in \eqref{eq:1} and \eqref{eq:2} have zero mean. Then \eqref{eq:derivative}, \eqref{eq:1} and \eqref{eq:2} yield
\begin{equation}\label{eq:3}
\frac{d}{d\epsilon}\E(T_\epsilon f)^\alpha=-2\alpha(\alpha-1)(1-2\epsilon)(\E f)^{\alpha-2}\sum_{i=1}^n\hat{f}(\{i\})^2+O((1-2\epsilon)^2).
\end{equation}
(The quantity $O(\cdot)$ in \eqref{eq:3} is a sum of finite terms involving $n, \epsilon, \alpha$ and $\hat{f}(A)$, $A\subseteq [n]$. Parseval's identity $\sum_{A}\hat{f}^2(A)=\E f^2=\E f$ implies the crude estimate $|\hat{f}(A)|\leq \sqrt{\E f}$. Given that $\E f$ is fixed, we can use this crude estimate to derive a uniform (with respect to $f$) bound on $O(\cdot)$ in \eqref{eq:3} in terms of $n, \epsilon, \alpha, \E f$). Hence, for $\epsilon=\epsilon(n)$ sufficiently close to 1/2, the function $\E(T_\epsilon f)^\alpha$ decreases faster if the degree-1 Fourier weight $W_1(f)=\sum_{i=1}^n\hat{f}(\{i\})^2$ is larger. Hence, for $\epsilon=\epsilon(n)$ close to 1/2, 
$\E(T_\epsilon f)^\alpha$ is maximized by some function $f$ with the largest degree-1 Fourier weight. This concludes the proof of the high noise case. 

Suppose that $f$ is a balanced Boolean function. Then, it is clear that the lexicographic function is just the dictator function. The function $2f-1$ is a balanced Boolean function on $\{-1, 1\}$. We have $W_1(f)=\frac{1}{4}W_1(2f-1)$ and $W_1(2f-1)$ is maximized when all Fourier weights of $2f-1$ are on degree 1. This occurs if and only if $2f-1=(-1)^{x_i}~\text{or}~-(-1)^{x_i}$ for some $i\in[n]$ (see, e.g., Proposition 2.50 of \cite{OD14}). This is equivalent to that $f$ is the dictator function up to isometric equivalence. This proves the statement for balanced Boolean functions.
\end{proof}

\begin{rmk}\label{rmk:nicd-noise-stability}
Recall that $\Psi_\alpha(x)=x^\alpha+(1-x)^\alpha$ for $\alpha>1$ and $x\in[0, 1]$. One can check that $\Psi_\alpha(x)$ is convex. Hence, $\E\Psi_\alpha(T_\epsilon f)$ is decreasing with respect to $\epsilon\in[0, 1/2]$. The same argument as before can be used to show that the statement of Theorem \ref{thm:high-low} also holds for $\E\Psi_\alpha(T_\epsilon f)$. 
In the low noise case, one can check that 
$$
\frac{d}{d\epsilon}\E(T_\epsilon (1-f))^\alpha\Big|_{\epsilon=0}=2\alpha\E ((1-f)^{\alpha-1}L(1-f)) =2\alpha\E (fLf)=-\frac{\alpha}{2}\cdot I(f),
$$
which, together with \eqref{eq:der-0}, yields that
$$
\frac{d}{d\epsilon}\E\Psi_\alpha(T_\epsilon (1-f))\Big|_{\epsilon=0}=-\alpha I(f).
$$
Analogous to \eqref{eq:3}, we have that for $\epsilon=\epsilon(n)$ close to 1/2,
$$
\frac{d}{d\epsilon}\E(T_\epsilon (1-f))^\alpha=-2\alpha(\alpha-1)(1-2\epsilon)(1-\E f)^{\alpha-2}\sum_{i=1}^n\hat{f}(\{i\})^2+O((1-2\epsilon)^2).
$$
This, together with \eqref{eq:3}, yields that for $\epsilon=\epsilon(n)$ close to 1/2, 
$$
\frac{d}{d\epsilon}\E\Psi_\alpha(f)=-2\alpha(\alpha-1)(1-2\epsilon)\big[(\E f)^{\alpha-2}+(1-\E f)^{\alpha-2}\big]\sum_{i=1}^n\hat{f}(\{i\})^2+O((1-2\epsilon)^2).
$$
\end{rmk}

As a consequence of Remark \ref{rmk:nicd-noise-stability}, we have the following result on the $k$-player NICD problem, which was proved by Mossel and O'Donnell \cite{MOD05} for balanced Boolean functions. (The assumption that functions are balanced does not seem to be used in their proof).

\begin{coro}\label{coro:app-nicd}
Let $0<\rho<1$. Let $Y^1, \cdots, Y^k\in\{0, 1\}^n$ be $\rho$-correlated uniformly random binary strings. Let $\E f$ be fixed. When $\rho=\rho(n)$ is sufficiently close to 1, the agreement probability $\P(f(Y^1)=\cdots=f(Y^k))$ is maximized by the lexicographic function. When $\rho=\rho(n)$ is sufficiently small, the agreement probability $\P(f(Y^1)=\cdots=f(Y^k))$ is maximized by some Boolean function $f$ with the largest degree-1 Fourier weight. 
\end{coro}

\begin{proof}
The statement readily follows from Remark \ref{rmk:nicd-noise-stability} and the following equation
$$
\P(f(Y^1)=\cdots=f(Y^k))=\E\Psi_k(T_\epsilon f).
$$
\end{proof}

We have the following heuristic for the two-player case when $\rho$ is close to 1. Suppose that $f$ is supported on $S$. Our goal is to maximize $\P(Y^1\in S, Y^2\in S)$, which is equivalent to the minimization of $\P(Y^1\in S, Y^2\notin S)$. Since $Y^1$ and $Y^2$ are $\rho$-correlated, we can think of $Y^2$ as obtained from $Y^1$ by flipping its coordinates independently with probability $(1-\rho)/2$.  When $\rho$ is close to 1, with high probability, $Y^1$ and $Y^2$ will differ by one bit, i.e., $(Y^1, Y^2)$ belongs to the edge boundary. Then smaller edge boundary implies larger agreement probability. Harper's theorem \cite{Har64} asserts that the lexicographic set has the least boundary among all sets with fixed size. Hence, the probability $\P(Y^1\in S, Y^2\in S)$ is maximized by the lexicographic function.

It is well-known that to determine  maximizers of degree-1 Fourier weight among Boolean functions with fixed mean is a hard question, and it is a folklore fact that the indicator of a Hamming ball is superior to the lexicographic function when the mean is sufficiently small (see e.g., \cite{KODW15}).  Let us include the explicit calculation below. Suppose that $\E f=2^{-m}$ for $1\leq m\leq n$. The lexicographic function $f(x)=\prod_{i=1}^mx_i$ is supported on a sub-cube $S$. 
Let $\rho=(1-2\epsilon)^2$. We have
$$
\P(f(Y^1)=f(Y^2))=4^{-n}(1+\rho)^{m}|S|.
$$
When $|S|$ is small, we let $g$ be a Boolean function supported on a vertex and $|S|-1$ vertices with Hamming distance 1 from that vertex. Elementary calculations yield
$$
\P(g(Y^1)=g(Y^2))=
4^{-n}(1+\rho)^{n-2}[(1-\rho)^2|S|^2+4\rho(2-\rho)|S|-4\rho].
$$
Then we have
\begin{align*}
\frac{d}{d\rho}\P(f(Y^1)=f(Y^2))\big|_{\rho=0} &=4^{-n}m|S|^2,\\
\frac{d}{d\rho}\P(g(Y^1)=g(Y^2))\big|_{\rho=0} &=4^{-n}[(n-4)|S|^2+8|S|-4].
\end{align*}
For $n-\log_2n<m\leq n-4$, we have
$$
\frac{d}{d\rho}\P(f(Y^1)=f(Y^2))\big|_{\rho=0}<\frac{d}{d\rho}\P(g(Y^1)=g(Y^2))\Big|_{\rho=0}.
$$
This implies that $\P(f(Y^1)=f(Y^2))<\P(g(Y^1)=g(Y^2))$ for small $\rho>0$.

Among balanced functions, the dictator function maximizes $\E(T_\epsilon f)^2$ at any noise level (see e.g., \cite{OD14}, Proposition 2.50). For balanced functions, we have
$$
\E f(Y^1)f(Y^2)=\E (1-f(Y^1))(1-f(Y^2)).
$$
Therefore, we have
$$
\P(f(Y^1)=f(Y^2))=2\E f(Y^1)f(Y^2),
$$
which is maximized by the dictator function. Similarly, we have 
$$
\P(f(Y^1)=f(Y^2)=f(Y^3))=3\E f(Y^1)f(Y^2)-1/2.
$$
Therefore, the dictator function is still the best strategy in the three-player case. This recovers Theorem 1.3 in \cite{MOD05}. We do not know if the dictator function also maximizes $\E(T_\epsilon f)^3$ among balanced Boolean functions.

We define the natural partial order relation on $\{0, 1\}^n$  as $x\preceq y$ if $x_i\leq y_i$ holds for all $i\in[n]$. A real-valued function $f$ on  $\{0, 1\}^n$ is called {\emph {monotone increasing}} if $f(x)\leq f(y)$, whenever $x\preceq y$, and $f$ is called {\emph {monotone decreasing}} if $f(x)\geq f(y)$, whenever $x\preceq y$. In both cases, we call the functions {\emph {monotone}}. 

\begin{thm}\label{thm:mono}
Let $\Phi$ be a convex function. For fixed mean $\E f$, the quantity $\E\Phi(T_\epsilon f)$ is maximized by some monotone function. 
\end{thm}

\begin{proof}
The proof is inspired by a shifting technique in \cite{Kle66} and a convex combination argument in \cite{CK14} (Theorem 3). Suppose that $f$ is supported on $S$. Let $S_2^n$ be the projection of $S$ on the last $n-1$ bits, i.e., $x_2^n\in S_2^n$ if $(0, x_2^n)\in S$ or $(1, x_2^n)\in S$. We define the following partition of $S_2^n$:
\begin{align*}
A &=\{x_2^n\in S_2^n: (0, x_2^n)\in S, (1, x_2^n)\in S\},\\
B &=\{x_2^n\in S_2^n: (0, x_2^n)\in S, (1, x_2^n)\notin S\},\\
C &=\{x_2^n\in S_2^n: (0, x_2^n)\notin S, (1, x_2^n)\in S\}.
\end{align*}
Then we have $S=(\{0, 1\}\times A)\cup (\{0\}\times B)\cup(\{1\}\times C)$.
Let $g$ be the Boolean function supported on $S'=(\{0, 1\}\times A)\cup (\{1\}\times \{B, C\})$. It is clear that $|S|=|S'|$, and that $f$ and $g$ have the same mean. We claim that $g$ is superior to $f$, i.e., $\E\Phi(T_\epsilon f)\leq \E\Phi(T_\epsilon g)$. Let $h$ be the Boolean function supported on $S''=(\{0, 1\}\times A)\cup (\{0\}\times \{B, C\})$. For any $x\in\{0, 1\}^n$, we will show that
\begin{equation}\label{eq:convex-comb}
T_\epsilon f(x)=\theta T_\epsilon g(x)+(1-\theta)T_\epsilon h(x),
\end{equation}
where $\theta$ depends on $x_2^n$. We only check this identity for $x=(0, x_2^n)$, since the argument is similar for $x=(1, x_2^n)$. Let $X\in\{0, 1\}^n$ be a uniformly random binary string. Let $Y=X+Z$, where the coordinates of $Z$ are i.i.d. Bernoulli($\epsilon$). Then we have
\begin{align*}
T_\epsilon f(0, x_2^n) &= \P(f(Y)=1|X=(0, x_2^n))\nonumber\\
&= \P(Y_2^n\in A|X=(0, x_2^n))\nonumber\\
& ~~~+\P(Y_1=0, Y_2^n\in B|X=(0, x_2^n))\nonumber\\
&~~~ +\P(Y_1=1, Y_2^n\in C|X=(0, x_2^n))\nonumber\\
&= \P(Y_2^n\in A|X=(0, x_2^n))\nonumber\\
& ~~~+\P(Y_1=1, Y_2^n\in B|X=(0, x_2^n))\nonumber\\
& ~~~+\P(Y_1=1, Y_2^n\in C|X=(0, x_2^n))\nonumber\\
&~~~+(1-2\epsilon)\P(Y_2^n\in B|X_2^n=x_2^n)\nonumber\\
&= T_\epsilon g(0, x_2^n)+(1-2\epsilon)\P(Y_2^n\in B|X_2^n=x_2^n).
\end{align*}
Similarly, we have
$$
T_\epsilon f(0, x_2^n)=T_\epsilon h(0, x_2^n)-(1-2\epsilon)\P(Y_2^n\in C|X_2^n=x_2^n).
$$
Therefore, identity \eqref{eq:convex-comb} holds with 
$$
\theta=\frac{\P(Y_2^n\in C|X_2^n=x_2^n)}{\P(Y_2^n\in B|X_2^n=x_2^n)+\P(Y_2^n\in C|X_2^n=x_2^n)}.
$$
Notice that $\theta$ is independent of $x_1$. We first apply the convex function $\Phi$ to \eqref{eq:convex-comb}, and then average both sides over the first bit. Then we have
\begin{equation}\label{eq:avg-1st}
\E\Phi(T_\epsilon f(X_1, x_2^n))\leq\theta\E\Phi(T_\epsilon g(X_1, x_2^n))+(1-\theta)\E\Phi(T_\epsilon h(X_1, x_2^n)).
\end{equation}
Notice that
$$
\E\Phi(T_\epsilon g(X_1, x_2^n))=\E\Phi(T_\epsilon h(X_1, x_2^n)),
$$
which follows from
\begin{align*}
T_\epsilon g(0, x_2^n)&=\P(Y_2^n\in A|X_1=0, X_2^n=x_2^n)\\
&~~~+\P(Y_1=1, Y_2^n\in \{B, C\}|X_1=0, X_2^n=x_2^n)\\
&=\P(Y_2^n\in A|X_2^n=x_2^n)\\
&~~~+\epsilon\P(Y_2^n\in \{B, C\}|X_2^n=x_2^n)\\
&=\P(Y_2^n\in A|X_2^n=x_2^n)\\
&~~~+\P(Y_1=0, Y_2^n\in \{B, C\}|X_1=1, X_2^n=x_2^n)\\
&=T_\epsilon h(1, x_2^n),
\end{align*}
and similarly
$$
T_\epsilon g(1, x_2^n)=T_\epsilon h(0, x_2^n).
$$
Hence, inequality \eqref{eq:avg-1st} becomes
$$
\E\Phi(T_\epsilon f(X_1, x_2^n))\leq\E\Phi(T_\epsilon g(X_1, x_2^n).
$$
We will have $\E\Phi(T_\epsilon f)\leq \E\Phi(T_\epsilon g)$ by averaging both sides of the above inequality over $x_2^n$. Repeat the argument over the last $n-1$ bits. We will arrive at a monotone function. 
\end{proof} 

\begin{rmk}
Theorem \ref{thm:mono} was proved in \cite{LM18} for $\Phi(x)=x^k$ where $k$ is a positive integer. In this case, the theorem can be rephrased as follows. Let $0<\rho<1$. Let $Y^1, \cdots, Y^k\in\{0, 1\}^n$ be $\rho$-correlated uniformly random binary strings. Fix the mean $\E f$. Then $\E \prod_{i=1}^k f(Y^{i})$ is maximized by some monotone function.
\end{rmk}

\begin{rmk}
It might be worth to point out a short proof of $\E\Phi(T_\epsilon f)\leq \E\Phi(T_\epsilon g)$. One can check that 
$$
T_\epsilon f(x)+T_\epsilon f(\sigma_1(x))=T_\epsilon g(x)+T_\epsilon g(\sigma_1(x)),
$$
$$
|T_\epsilon f(x)-T_\epsilon f(\sigma_1(x))|\leq|T_\epsilon g(x)-T_\epsilon g(\sigma_1(x))|.
$$
Then the desired statement follows from the majorization inequality for convex functions.
\end{rmk}

For $k\geq 2$, $\Phi(x)=x^k+(1-x)^k$ is convex for $x\in[0, 1]$. As a consequence of Theorem \ref{thm:mono}, we have the following result on the $k$-player NICD problem, which was obtained by Mossel and O'Donnell \cite{MOD05} for balanced Boolean function. (The assumption that functions are balanced does not seem to be used in their proof).
\begin{coro}
Let $0<\rho<1$. Let $Y^1, \cdots, Y^k\in\{0, 1\}^n$ be $\rho$-correlated uniformly random binary strings. Fix the mean $\E f$. Then $\P(f(Y^1)=\cdots=f(Y^k))$ is maximized by some monotone function. 
\end{coro}

We have seen from Theorem \ref{thm:high-low} that among balanced Boolean functions the dictator function maximizes $\E(T_\epsilon f)^\alpha$ in both the low and high noise scenarios for fixed $n$ and $\alpha$. One may expect that the same property holds for arbitrary noise. The following result shows that this is not true if $\alpha$ is large. 

We define the \emph {simplicial order} on $\{0, 1\}^n$ as follows. Each $x\in \{0, 1\}^n$ is associated with a subset $A_x\subseteq [n]$ in the natural way that $i\in A_x$ if and only if $x_i=1$.  We say $x\prec y$ if $|A_x|>|A_y|$ or $|A_x|=|A_y|$ but $\max(A_x\Delta A_y)\in A_y$, where $\Delta$ is the symmetric difference operation between two sets.

For an odd number $n$, we define the {\emph {majority function}} 
$$
\text{Maj}_n(x)=\frac{1+\text{sgn}(\sum_{i=1}^nx_i-n/2)}{2}
$$
In particular, $\text{Maj}_1(x)$ is the dictator function, which only looks at the first bit.

\begin{thm}\label{thm:asym}
Let $n, \epsilon$ be fixed. Let $\alpha$ be sufficiently large. Then, $\E(T_\epsilon f)^\alpha$ is maximized by the Boolean function supported on the initial segment of  $\{0, 1\}^n$ labelled in the simplicial order. In particular, among balanced Boolean functions, $\E(T_\epsilon f)^\alpha$ is maximized by any function which is 0 on all strings with fewer than $n/2$ 1's; and, for $n$ odd, $\E(T_\epsilon f)^\alpha$ is maximized by the majority function.
\end{thm}

The following statement was informed to us by Nathan Keller. It is probably scattered somewhere in the literature. We provide a proof for the convenience of readers. 

\begin{lem}\label{lem:max-1}
If $f$ is a monotone increasing (monotone decreasing, respectively) function (not necessarily Boolean), then $T_\epsilon f(x)$ is also monotone increasing (monotone decreasing, respectively). In particular, $T_\epsilon f(x)$ is maximized at $x=\vec{1}=(1, \cdots, 1)$ ($x=\vec{0}=(0, \cdots, 0)$, respectively).
\end{lem}

\begin{proof}
We only prove the monotone increasing case, since the other case can be proved in a similar manner. It suffices to show that $T_\epsilon f(x)\geq T_\epsilon f(x')$, where $x=(1, x_2^n)$ and $x'=(0, x_2^n)$, i.e., $x$ and $x'$ only differ on 1 bit. Recall that
$$
T_\epsilon f(x)=\sum_{y\in\{0, 1\}^n} \epsilon^{d(x, y)}(1-\epsilon)^{n-d(x, y)}f(y),
$$
where $d(x, y)$ is the Hamming distance between $x$ and $y$. Couple the summands on $y=(1, y_2^n)$ and $y'=(0, y_2^n)$. We can rewrite $T_\epsilon f(x)$ as
$$
T_\epsilon f(x)=\sum_{y_2^n\in\{0, 1\}^{n-1}} \epsilon^{d(x_2^n, y_2^n)}(1-\epsilon)^{n-1-d(x_2^n, y_2^n)}[(1-\epsilon)f(1, y_2^n)+\epsilon f(0, y_2^n)],
$$
where the summation is taken over all binary sequences of length $n-1$. Similarly, we have
$$
T_\epsilon f(x')=\sum_{y_2^n\in\{0, 1\}^{n-1}} \epsilon^{d(x_2^n, y_2^n)}(1-\epsilon)^{n-1-d(x_2^n, y_2^n)}[\epsilon f(1, y_2^n)+(1-\epsilon) f(0, y_2^n)].
$$
Since $f$ is monotone increasing, we have $ f(1, y_2^n)\geq f(0, y_2^n)$. The fact that $\epsilon\in[0, 1/2]$ implies 
$$
(1-\epsilon)f(1, y_2^n)+\epsilon f(0, y_2^n)\geq \epsilon f(1, y_2^n)+(1-\epsilon) f(0, y_2^n).
$$
Hence, the desired statement follows.
\end{proof}

We need the following result, which was essentially proved in \cite{MOD05}.

\begin{lem}[Proposition 4.2, \cite{MOD05}]\label{lem:maj}
The function $T_\epsilon f(\vec{1})$ is maximized by the Boolean function supported on the initial segment of  $\{0, 1\}^n$ labelled in the simplicial order. In particular, among balanced Boolean function, $T_\epsilon f(\vec{1})$ is maximized by any function which is 0 on all strings with fewer than $n/2$ 1's; and, for $n$ odd, $T_\epsilon f(\vec{1})$ is maximized by the majority function. 
\end{lem}
\begin{proof}
The statement simply follows from
$$
T_\epsilon f(\vec{1})=\sum_{x\in S}\epsilon^{d(x, \vec{1})}(1-\epsilon)^{n-d(x, \vec{1})},
$$
where $S$ is the support of $f$, and $d(x, \vec{1})$ is the Hamming distance between $x$ and $\vec{1}$, and the simple fact that the quantity being summed is strictly decreasing with respect to $d(x, \vec{1})$.
\end{proof}

\begin{proof}[Proof of Theorem \ref{thm:asym}]
The proof relies on the simple observation that $\E(T_\epsilon f)^\alpha$ is essentially determined by the largest value of $T_\epsilon f$ when $\alpha$ is large. To avoid ambiguity, we assume that the support of $f$ has size $\sum_{i=1}^k{n\choose i}$ for some $k$. One can apply the same argument in the general setting. Invoke Theorem \ref{thm:mono}, then we can assume that $f$ is monotone increasing. Using Lemma \ref{lem:max-1}, we have
$$
\E(T_\epsilon f)^\alpha=2^{-n}\sum_{x\in\{0, 1\}^n}(T_\epsilon f(x))^\alpha\leq (T_\epsilon f(\vec{1}))^\alpha.
$$
Let $g$ be the Boolean function supported on the initial segment of  $\{0, 1\}^n$ labelled in the simplicial order, which is the Hamming ball centered at $\vec{1}$ with radius $k$. It is clear that
$$
\E(T_\epsilon g)^\alpha= 2^{-n}\sum_{x\in\{0, 1\}^n}(T_\epsilon g(x))^\alpha\geq 2^{-n}(T_\epsilon g(\vec{1}))^\alpha.
$$
By Lemma \ref{lem:maj}, we have $T_\epsilon f(\vec{1})<T_\epsilon g(\vec{1})$ for $f\neq g$. (Here, we implicitly use the assumption of the size of the support of $f$). Then the theorem follows from $(T_\epsilon f(\vec{1}))^\alpha<2^{-n}(T_\epsilon g(\vec{1}))^\alpha$ for sufficiently large $\alpha=\alpha(n)$. 
\end{proof}

Then we can recover the following result of Mossel and O'Donnell \cite{MOD05}.
\begin{coro}
Let $0<\rho<1$. Let $Y^1, \cdots, Y^k\in\{0, 1\}^n$ be $\rho$-correlated uniformly random binary strings. For sufficiently large $k$, among balanced Boolean functions, the agreement probability $\P(f(Y^1)=\cdots=f(Y^k))$ is maximized by any function which is 0 on all strings with fewer than $n/2$ 1's. For $n$ odd, the agreement probability is maximized by the majority function.
\end{coro}

\begin{rmk}
We have shown that, within the class of balanced Boolean functions, the dictator function $\text{Maj}_1$ and the majority function $\text{Maj}_n$ have the maximal noise stability in the asymptotic regimes--$\epsilon$ close to 0 or 1/2 with $k$ fixed, and $k$ large with $\epsilon$ fixed, respectively.
But there exists $k, \epsilon$, $n$ odd and $1<r<n$ such that the function $\text{Maj}_r(x)=[1+\text{sgn}\left(\sum_{i=1}^rx_i-r/2\right)]/2$ is superior to both the dictator function and the majority function. Consider the numerical example $k=10, \epsilon=0.26, n=5, r=3$, which is taken from  \cite{MOD05} (Proposition 5.2). One can check that $\E (T_\epsilon\text{Maj}_1)^{10}\leq 0.0247, \E (T_\epsilon\text{Maj}_5)^{10}\leq0.0244$ and $\E (T_\epsilon\text{Maj}_3)^{10}\geq0.0248$. We do not know whether $\E(T_\epsilon f)^k$ is always maximized by some $\text{Maj}_r$.
\end{rmk}


\section{The most informative Boolean function}\label{sec:mutual}
Let $X\in\{0, 1\}^n$ be a binary string selected uniformly at random. Let $Y$ be the output of $X$ through a $\text{BSC}(\epsilon)$ channel, i.e., $Y=X+Z$, where the coordinates of $Z$ are independent Bernoulli($\epsilon$). Let $f$ be a Boolean function. It is  conjectured by Courtade and Kumar {\cite{CK14} that the dictator function maximizes the mutual information $I(X; f(Y))$ between $X$ and $f(Y)$. Recall that
$$
I(X; f(Y))=H(f(Y))-H(f(Y)|X).
$$
Notice that $f(Y)$ is a Bernoulli random variable with the parameter $\P(f(Y)=1)=\E f$. We have $H(f(Y))=H(\E f)$. Here, we denote by $H(p)$ the Shannon entropy of $\text{Bernoulli}(p)$. Hence, for fixed $\E f$, it suffices to maximize $-H(f(Y)|X)$. Given $X$, the random variable $f(Y)$ is still Bernoulli with the parameter $\E[f(Y)|X]=T_\epsilon f(X)$. Therefore, we have
$$
-H(f(Y)|X)=\E\Phi(T_\epsilon f),
$$
where the entropy function $\Phi(x)=x\log x+(1-x)\log(1-x)$ for $x\in [0, 1]$.

Recall that $\Psi_\alpha(x)=x^\alpha+(1-x)^\alpha$ for $\alpha>1$ and $x\in[0, 1]$. One can check that $\partial_\alpha\Psi_\alpha(x)|_{\alpha=1}=\Phi(x)$. Hence, we have
$$
-H(f(Y)|X)=\partial_\alpha\E\Psi_\alpha(T_\epsilon f)|_{\alpha=1}.
$$ 
The initial value $\E\Psi_\alpha(T_\epsilon f)|_{\alpha=1}$ is fixed. Hence, for fixed $\E f$, if there a unique (up to isometric equivalence) maximizer $f$ of $\E\Psi_\alpha(T_\epsilon f)$  for all $\alpha\in (1, \alpha_0)$, where $\alpha_0>1$ could be dimension dependent, then $f$ also maximizes $I(X; f(Y))$ among all Boolean functions with the same expectation $\E f$. Conversely, for fixed $\E f$, if $f$ is the unique (up to isometric equivalence) maximizer  of $I(X; f(Y))$, then it also maximizes $\E\Psi_\alpha(T_\epsilon f)$ for $\alpha\in (1, \alpha_0)$ among all Boolean functions with the same expectation $\E f$. This connection between Courtade-Kumar's conjecture and the $\alpha$-NICD problem, together with Remark \ref{rmk:nicd-noise-stability}, yields the following result.  
\begin{coro}
Let $\E f$ be fixed. When $\epsilon=\epsilon(n)$ is sufficiently small, the mutual information $I(X; f(Y))$ is maximized by the lexicographic function. When $\epsilon=\epsilon(n)$ is sufficiently close to $1/2$,  the mutual information $I(X; f(Y))$ is maximized by some Boolean function with the largest degree-1 Fourier weight. In particular, within the class of balanced Boolean functions, the dictator function maximizes  the mutual information in both scenarios.   
\end{coro}

\begin{rmk}
In Appendix B of  \cite{CK14}, Courtade and Kumar observed that their conjecture holds when $\epsilon=\epsilon(n)\to 0, 1/2$. A quantitative bound in the high noise case can be found in Corollary 1 of Ordentlich, Shayevitz, and Weinstein \cite{OSW15}.(In an unpublished work, similar results were independently obtained by Sachdeva and Samorodnitsky). Samorodnitsky \cite{Sam16} gave a dimension-free bound in the high noise setting. Our result provides a finer characterization of maximizers of $I(X; f(Y))$ when $\E f$ is fixed. This was also observed in Appendix B of \cite{CK14}. 
\end{rmk}

\begin{rmk}
As a consequence of Theorem \ref{thm:mono}, it suffices to study Courtade-Kumar's conjecture for monotone functions. This has been observed by Courtade and Kumar \cite{CK14}, and Huleihel and Ordentlich \cite{HO17}.
\end{rmk}

When $\alpha=1$ and $\alpha=2$, the dictator function is the maximizer of $\E\Psi_\alpha(T_\epsilon f)$ within the class of balanced Boolean functions. It is reasonable to expect that the dictator function still plays the extremal role for any $1<\alpha<2$. So we propose the following conjecture, which implies Courtade-Kumar's conjecture for balanced Boolean functions.
\begin{conj}\label{conj:phi-entropy}
For $1\leq \alpha\leq 2$, the dictator function maximizes $\E(T_\epsilon f)^\alpha$ within the class of balanced Boolean functions.
\end{conj}

Let $\alpha\geq 1$. The minimum of $\E(T_\epsilon f)^\alpha$ may be achieved by a Boolean function whose support is ``evenly spaced'' in the discrete cube.  It is likely that the dictator function still plays the extremal role among all functions $f:\{0,1\}^n\to [0, 1]$ such that $\E f=1/2$. It is clear that among all functions $f:\{0,1\}^n\to [0, 1]$ such that $\E f=1/2$, the functional $\E(T_\epsilon f)^\alpha$ is minimized by the constant function $f=1/2$. Without assuming boundedness, the maximum is achieved by Dirac's delta function, which is supported on a single vertex.

\begin{rmk}
Let $\Phi$ be a convex function. The $\Phi$-entropy of a function $f:\{0, 1\}^n\to \R$ is defined as $H_\Phi(f)=\E\Phi(f)-\Phi(\E f)$ (see \cite{Cha04} for discussions of $\Phi$-entropy in more general settings). Let $\Phi(x)=x\log x+(1-x)\log(1-x)$. The Courtade-Kumar conjecture can be rephrased as that the dictator function maximizes $H_\Phi(T_\epsilon f)$ among all Boolean functions. We considered the function $\Phi(x)=x^\alpha$ for $1<\alpha<2$ and conjectured that the dictator function is the maximizer of $H_\Phi(T_\epsilon f)$ within the class of balanced Boolean functions. 
Anantharam et al. \cite{ABCJN17} conjectured that the dictator function is still the maximizer for $H_\Phi(T_\epsilon f)$ with the convex function $\Phi(x)=1-2\sqrt{x(1-x)}$, which is the squared Hellinger distance between two Bernoullis with parameters $x$ and $1-x$, respectively. 
\end{rmk}


\section{General models}
\label{sec:extensions}

Now we discuss the problem of noise stability in more general contexts, where algebraic and geometric features will show their impacts on the analysis. 

\subsection{Discrete torus}


In this subsection, we discuss noise stability of Boolean functions defined on the discrete torus $(\Z/p\Z)^n$, where $\Z/p\Z=\{0, 1, \cdots, p-1\}$ is the cyclic group of order $p$ (and $p$ is not necessarily a prime).  

We first give a brief introduction of Fourier analysis on the group $(\Z/p\Z)^n$ associated with the uniform measure $\mu$. We define $e_p(t)=e^{i2\pi t/p}$ for $t\in\R$. One can check that the set of functions $\{e_p(\xi\cdot x)\}_{\xi\in (\Z/p\Z)^n}$ forms an orthonormal basis, where $\xi\cdot x=\xi_1x_1+\cdots+\xi_nx_n$. Hence, any function $f: (\Z/p\Z)^n\to \R$ has the following Fourier representation
\begin{equation}\label{eq:fourier-expan-1}
f(x)=\sum_{\xi\in (\Z/p\Z)^n}\hat{f}(\xi)e_p(\xi\cdot x),
\end{equation}
where Fourier coefficients  $\hat{f}(\xi)=\E f(x)e_p(\xi\cdot x)$.

Analogous to Definition \ref{defn:noise}, we 
define noise operator acting on functions defined on the discrete torus in a general way  without specifying the distribution of the noise. We will discuss the problem of  noise stability under two types of noise later.

\begin{defn}\label{defn:noise-1}
Let $0\leq \epsilon\leq 1-1/p$. The noise operator $T_\epsilon$ acting on $f:(\Z\slash p\Z)^n\to \R$ is defined as follows
$$
T_\epsilon f(x)=\E f(x+Z),
$$
where  $Z=(Z_1, \cdots, Z_n)\in (\Z/p\Z)^n$ is a random vector with i.i.d. coordinates.
\end{defn}

Correspondingly, the NICD problem can be stated as follows.  Let $X\in(\Z/p\Z)^n$ be a uniform random vector; that is, the coordinates of $X$ are independent and uniform on $\Z/p\Z$.  We pass it on to $k$ players through independent additive noise channels, which are represented as independent copies of $Z$. Upon receiving the message, each player applies a Boolean function to output one alphabet. As usual, their goal is to maximize the agreement probability. We denote by $Y^1, \cdots, Y^k$ the $k$ corrupted versions of $X$. The NICD problem asks the maximum of $\P(f_1(Y^1)=\cdots=f_k(Y^k))$, where $f_1, \cdots, f_k$ are Boolean functions. One can check that Proposition \ref{prop:holder} still holds in this multi-alphabet setting, i.e., the $k$ players should apply the same Boolean function.

Similar to the binary case, the problem of NICD in the multi-alphabet setting also has close connection with the problem of $\alpha$-stability with $\alpha=k$. As in the binary case, the analysis of $\alpha$-stability needs the notation of influence. Influence of real-valued functions (not necessarily Boolean) can be defined in general domains (see e.g., Definition 8.22, \cite{OD14}). We adapt Definition \ref{defn:influence} as follows. Let $\tilde{Z}_j$ be a the restriction of $Z_j$ on $\Z_p\backslash\{0\}$; that is, it has distribution
\begin{equation}\label{eq:z tilde}
\P\big(\tilde{Z}_j=\ell\big)=\frac{\P(Z_j=\ell)}{1-\P(Z_j=0)}.
\end{equation}
Analogous to \eqref{eq:flipping}, we define the random flipping operator $\tilde{\sigma}_j$ as follows
\begin{equation}\label{eq: random flipping}
\tilde{\sigma}_j(x_1,\cdots, x_j, \cdots, x_n)=(x_1,\cdots, x_j+\tilde{Z}_j, \cdots, x_n).
\end{equation}

Then we define influence of Boolean functions on the discrete torus as follows.

\begin{defn}\label{defn:influence-1}
Let $f:(\Z\slash p\Z)^n\to\{0, 1\}$ be a Boolean function. The influence of the $j$-th variable $I_j(f)$ is defined as
$$
I_j(f)=\P(f(X)\neq f(\tilde{\sigma}_j(X))).
$$
(We assume that $\tilde{Z}_j$ is independent of $X$).
The total influence $I(f)$ is defined as 
$$
I(f)=\sum_{j=1}^nI_j(f).
$$ 
\end{defn}

\subsubsection{Noise: type I}

One type of the noise distributions is defined as follows. Let $Z=(Z_1, \cdots, Z_n)$ be the noise vector with i.i.d. coordinates. We define
\begin{equation}\label{eq:noise-1}
\P(Z_1=\ell)=
\left\{  \begin{array}{ll}
1-\epsilon, & \ell=0\\
\frac{\epsilon}{p-1}, & \ell\neq 0.
    \end{array} \right.
\end{equation}
In other words, the additive noise channel $Z$ preserves the value of an alphabet with probability $1-\epsilon$ and changes its value to other values equally likely.

It is easy to check that $\sum_{j=0}^{p-1}e_p(jk)=0$ for any $k\neq0$. This identity, together with the Fourier representation \eqref{eq:fourier-expan-1}, allows us to write noise operator in Definition \ref{defn:noise-1} as follows
\begin{equation}\label{eq:noise-fourier-1}
T_\epsilon f(x)=\sum_{\xi\in (\Z/p\Z)^n}\left(1-\frac{p\epsilon}{p-1}\right)^{|\supp(\xi)|}\hat{f}(\xi)e_p(\xi\cdot x),
\end{equation}
where $\supp(\xi)=\{j: \xi_j\neq 0\}$. When $p=2$, this Fourier representation coincides with \eqref{eq:noise-fourier}.

It seems that our next result could follow from a general result, Proposition 8.23 in \cite{OD14}.

\begin{prop}\label{prop:f-i}
Let $f:(\Z\slash p\Z)^n\to\{0, 1\}$ be a Boolean function. Then we have
$$
I_j(f)=\frac{2p}{p-1}\sum_{\xi: \xi_j\neq0}|\hat{f}(\xi)|^2,
$$
and 
$$
I(f)=\frac{2p}{p-1}\sum_{\xi\in (\Z/p\Z)^n}|\supp(\xi)||\hat{f}(\xi)|^2.
$$

\end{prop}

\begin{proof}
Since $f$ takes values 0 or 1, one can rewrite $I_i(f)$ as
\begin{equation}\label{eq:I_j}
I_j(f)=\E(f(X)-f(\tilde{\sigma}_j(X)))^2.
\end{equation}
Notice that both $X$ and $\tilde{\sigma}_j(X)$ are uniformly random. By Parseval's identity, we have
\begin{equation}\label{eq:ef2}
\E f(X)^2=\E f(\tilde{\sigma}_j(X)))^2=\sum_{\xi\in (\Z/p\Z)^n}|\hat{f}(\xi)|^2.
\end{equation}
Using the Fourier representation, we have
$$
\E f(X)f(\tilde{\sigma}_j(X)))=\E\sum_{\xi, \eta\in (\Z/p\Z)^n}\hat{f}(\xi)\overline{\hat{f}(\eta)}e_p((\xi-\eta)\cdot X)e_p(-\eta_j\tilde{Z}_j),
$$
where $\overline{\hat{f}(\eta)}$ is the complex conjugate of $\hat{f}(\eta)$.
Since $\tilde{Z}_j$ and $X$ are independent, we have
$$
\E e_p((\xi-\eta)\cdot X)e_p(-\eta_j\tilde{Z}_j)=\E e_p((\xi-\eta)\cdot X)\E e_p(-\eta_j\tilde{Z}_j).
$$
Owing to the orthogonality, $\E e_p((\xi-\eta)\cdot X)$ vanishes if $\xi\neq\eta$. 
One can check that 
$$
\E e_p(-\eta_j\tilde{Z}_j)=
\begin{cases}
1, &\eta_j=0\\ 
-\frac{1}{p-1},  &\eta_j\neq0.
\end{cases}
$$
Therefore, we have
\begin{equation}\label{eq:cross}
\E (f(X)f(\tilde{\sigma}_j(X)))=\sum_{\xi: \xi_j=0}|\hat{f}(\xi)|^2-\frac{1}{p-1}\sum_{\xi: \xi_j\neq0}|\hat{f}(\xi)|^2.
\end{equation}
The desired statement follows from \eqref{eq:I_j}, \eqref{eq:ef2} and \eqref{eq:cross}.
\end{proof}

We now show an analogue of Theorem \ref{thm:high-low}. 

\begin{thm}\label{thm:high-low-1}
Let $\alpha>1$. Let $\E f$ be fixed. When $\epsilon=\epsilon(n)$ is sufficiently small, the quantity $\E(T_\epsilon f)^\alpha$ is maximized by some Boolean function with the least total influence. When $\epsilon=\epsilon(n)$ is sufficiently close to $1-1/p$, the quantity $\E(T_\epsilon f)^\alpha$ is maximized by some Boolean function with the largest degree-1 Fourier weight $W_1(f)=\sum_{\xi: |\supp(\xi)|=1}|\hat{f}(\xi)|^2$.
\end{thm} 
\begin{proof}
The statement can be proved in a manner similar to that of Theorem \ref{thm:high-low}. We only give a sketch. The function $\E(T_\epsilon f)^\alpha$ is  decreasing for $\epsilon\in[0, 1-1/p]$. We have fixed initial value $\E(T_0 f)^\alpha=\E f$ and ending value $\E(T_{1-1/p} f)^\alpha=(\E f)^\alpha$. In the low noise case, the equation 
$$
\frac{d}{d\epsilon}\E(T_\epsilon f)^\alpha\big|_{\epsilon=0}=-\alpha I(f)/2
$$ 
still holds with total influence $I(f)$ given in Definition \ref{defn:influence-1}. When $\epsilon=\epsilon(n)$ is close to $1-1/p$, one can check that the leading term of $\frac{d}{d\epsilon}\E(T_\epsilon f)^\alpha$ is
$$
-\alpha(\alpha-1)(\E f)^{\alpha-2}\frac{p}{p-1}\left(1-\frac{p\epsilon}{p-1}\right)\sum_{\xi: |\supp(\xi)|=1}|\hat{f}(\xi)|^2.
$$
Then the statement easily follows.
\end{proof}

The following is an analogy of Theorem \ref{thm:mono}

\begin{thm}\label{thm:mono-1}
Let $\Phi$ be a convex function. For fixed mean $\E f$, the quantity $\E\Phi(T_\epsilon f)$ is maximized by some monotone function. 
\end{thm}

\begin{proof}
We only need to slightly modify the proof of Theorem \ref{thm:mono}.
Suppose that $f$ is supported on $S$. For each pair $j, k\in\Z\slash p\Z$ such that $j<k$, we define
\begin{align*}
B_{j, k} &=\{x_2^n\in (\Z\slash p\Z)^{n-1}: (j, x_2^n)\in S, (k, x_2^n)\notin S\},\\
C_{j, k} &=\{x_2^n\in (\Z\slash p\Z)^{n-1}: (j, x_2^n)\notin S, (k, x_2^n)\in S\}.
\end{align*}
Let $A_{j, k}=S\backslash((\{j\}\times B_{j, k})\cup (\{k\}\times C_{j, k}))$.
Let $g_{j, k}$ be the Boolean function supported on $S_{j, k}'=A_{j, k}\cup (\{k\}\times (B_{j, k}\cup C_{j, k}))$. It is clear that $|S|=|S_{j, k}'|$,  and that $f$ and $g_{j, k}$ have the same mean. We claim that $g_{j, k}$ is superior to $f$, i.e., $\E\varphi(T_\epsilon f)\leq \E\varphi(T_\epsilon g_{j, k})$. Let $h_{j, k}$ be the Boolean function with support $S_{j, k}''=A_{j, k}\cup (\{j\}\times (B_{j, k}\cup C_{j, k}))$. For any $x\in(\Z\slash p\Z)^n$, the following identity still holds
\begin{equation}\label{eq:convex-comb-1}
T_\epsilon f(x)=\theta T_\epsilon g(x)+(1-\theta)T_\epsilon h(x),
\end{equation}
where $\theta$ depends on $x_2^n$. For $x=(j, x_2^n)$ we have
\begin{align*}
T_\epsilon f(x) &= \P(Y\in A_{j, k}|X=(j, x_2^n))\nonumber\\
& ~~~+\P(Y_1=j, Y_2^n\in B_{j, k}|X=(j, x_2^n))\nonumber\\
& ~~~+\P(Y_1=k, Y_2^n\in C_{j, k}|X=(j, x_2^n))\nonumber\\
&= \P(Y\in A_{j, k}|X=(j, x_2^n))\nonumber\\
& ~~~+\P(Y_1=k, Y_2^n\in B_{j, k}|X=(j, x_2^n))\nonumber\\
& ~~~+\P(Y_1=k, Y_2^n\in C_{j, k}|X=(j, x_2^n))\nonumber\\
&~~~+(1-p\epsilon/(p-1))\P(Y_2^n\in B_{j, k}|X_2^n=x_2^n)\nonumber\\
&= T_\epsilon g_{j, k}(x)+\left(1-\frac{p\epsilon}{p-1}\right)\P(Y_2^n\in B_{j, k}|X_2^n=x_2^n).
\end{align*}
Similarly, we have
$$
T_\epsilon f(x)=T_\epsilon h_{j, k}(x)-\left(1-\frac{p\epsilon}{p-1}\right)\P(Y_2^n\in C_{j, k}|X_2^n=x_2^n).
$$
Therefore, identity \eqref{eq:convex-comb-1} holds with 
$$
\theta=\frac{\P(Y_2^n\in C_{j, k}|X_2^n=x_2^n)}{\P(Y_2^n\in B_{j, k}|X_2^n=x_2^n)+\P(Y_2^n\in C_{j, k}|X_2^n=x_2^n)}.
$$
The case  $x=(k, x_2^n)$ can be checked in the same manner. When $x_1\neq j, k$, we have $T_\epsilon f(x)=T_\epsilon g_{j, k}(x)=T_\epsilon h_{j, k}(x)$. Hence, we first apply the convex function $\Phi$ to \eqref{eq:convex-comb-1}, and then average both sides over the first bit. Then we have
\begin{equation}\label{eq:avg-1st-1}
\E\Phi(T_\epsilon f(X_1, x_2^n))\leq\theta\E\Phi(T_\epsilon g_{j, k}(X_1, x_2^n))+(1-\theta)\E\Phi(T_\epsilon h_{j, k}(X_1, x_2^n)).
\end{equation}
Similarly, we have
$$
\E\Phi(T_\epsilon g_{j, k}(X_1, x_2^n))=\E\Phi(T_\epsilon h_{j, k}(X_1, x_2^n)),
$$
which follows from
$$
T_\epsilon g_{j, k}(j, x_2^n)=T_\epsilon h(k, x_2^n).
$$
$$
T_\epsilon g_{j, k}(k, x_2^n)=T_\epsilon h(j, x_2^n).
$$
and that $T_\epsilon g_{j, k}(x)=T_\epsilon h_{j, k}(x)$ for $x_1\neq j, k$. Then inequality \eqref{eq:avg-1st-1} becomes
$$
\E\Phi(T_\epsilon f(X_1, x_2^n))\leq\E\Phi(T_\epsilon g_{j, k}(X_1, x_2^n)).
$$
We will have $\E\Phi(T_\epsilon f)\leq \E\Phi(T_\epsilon g_{j, k})$ by averaging over $x_1^n$. Repeat the argument for all such pairs $(j, k)$ and the last $n-1$ coordinates. We will arrive at a monotone function. 
\end{proof}

\subsubsection{Noise: type II}

In some sense, our results in the previous sub-subsection rely on the algebraic or group feature of the discrete torus. This is also the reason why, under type-I noise, we do not have a geometric interpretation of the total influence for $p>2$. This motivates us to consider another type of noise. We adapt the noise distribution defined in \eqref{eq:noise-1} as follows
\begin{equation}\label{eq:noise-2}
\P(Z_1=\ell)=
\left\{  \begin{array}{ll}
1-\epsilon, & \ell=0\\
\epsilon/2, & \ell=1, p-1.
    \end{array} \right.
\end{equation}
Since $-1=p-1$ in $\Z\slash p\Z$, the above noise only changes an alphabet to its nearest values. Analogous to \eqref{eq:noise-fourier-1}, we have the following Fourier representation 
\begin{equation}\label{eq:f-noise-1}
T_\epsilon f(x)=\sum_{\xi\in (\Z/p\Z)^n}\prod_{j=1}^n[1-\epsilon(1-\cos(2\pi\xi_j/p))]\hat{f}(\xi)e_p(\xi\cdot x).
\end{equation}

Recall our definition of $\tilde{Z}_j$ in \eqref{eq:z tilde}. Under the noise in \eqref{eq:noise-2}, $\tilde{Z}_j$ is a Bernoulli random variable taking $1$ and $-1$ with equal probability. 
In this case, we can connect the influence in Definition \ref{defn:influence-1} to edge boundary as in the discrete cube setting.
Let $S$ be the support of $f$. We define the $j$-th direction edge boundary
$$
\partial_jS=\{(x, y): x_j-y_j\in\{\pm 1\}, x_k=y_k~\text{for}~k\neq j\}
$$
and the edge boundary $\partial S=\cup_{j=1}^n\partial_jS$. Analogous to \eqref{eq:geo-i-influence} and \eqref{eq:geo-influence}, we have the following relation between edge boundary and influence
\begin{align} 
I_j(f) &=\frac{|\partial_j S|}{p^n}, \label{eq:geo-i-influence-1}\\
I(f) &=\frac{|\partial S|}{p^n}. \label{eq:geo-influence-1}
\end{align}
An alert reader may have noticed that, taking $p=2$,  identities \eqref{eq:geo-i-influence-1} and \eqref{eq:geo-influence-1} do not match \eqref{eq:geo-i-influence} and \eqref{eq:geo-influence}, respectively. This is because, for $p=2$, our definition \eqref{eq:noise-2} does not yield a probability distribution, since $\epsilon/2$ mass is missing. 

Analogous to Proposition \ref{prop:f-i}, we have the following Fourier representation of influence. 
\begin{prop}\label{prop:f-i-1}
Let $f:(\Z\slash p\Z)^n\to\{0, 1\}$ be a Boolean function. Then we have
\begin{align*}
I_j(f) &=2\sum_{\xi\in (\Z/p\Z)^n}(1-\cos(2\pi\xi_j/p))|\hat{f}(\xi)|^2,\\
I(f) &=2\sum_{\xi\in (\Z/p\Z)^n}\sum_{j=1}^n(1-\cos(2\pi\xi_j/p))|\hat{f}(\xi)|^2.
\end{align*}
\end{prop}

The following statement can be proved in the same manner as that of Theorem \ref{thm:high-low}.

\begin{thm}\label{thm:high-low-2}
Let $\alpha>1$. Let $\E f$ be fixed. When $\epsilon=\epsilon(n)$ is sufficiently small, the quantity $\E(T_\epsilon f)^\alpha$ is maximized by some Boolean function with the least total influence, i.e., some Boolean function supported on a set with the least edge boundary.
\end{thm}

\begin{rmk}
Bollob\'{a}s and Leader \cite{BL91} proved sharp edge isoperimetric inequalities for the discrete torus and the grid (Theorem 8 and Theorem 3, respectively). When the subset possesses certain type of cardinalities, they have characterization of the extremal set; but, in general, they do not know which set to take, although they know the sharp bound of the edge boundary of the extremal sets.   
\end{rmk}

\begin{rmk}
Theorem \ref{thm:high-low-1} and Theorem  \ref{thm:high-low-2} characterize maxmizers in Fourier analytic and geometric ways, respectively. This difference results from that noise operator in Definition \ref{defn:noise-1} under two type of noises \eqref{eq:noise-1} and \eqref{eq:noise-2} captures algebraic/group and geometric/graphic features of discrete torus, respectively.
\end{rmk}

\begin{rmk}
We have the following analogue of Theorem \ref{thm:high-low-2} for general Markov semi-groups $(P_t)_{t\geq0}$ defined in \eqref{eq:semi-group}. When $t>0$ is sufficiently small, $\E(P_t f)^\alpha$ is maximized by some Boolean function supported on a set with the least edge boundary. This follows from the relation
$$
\frac{d}{dt}\E(P_t f)^\alpha\Big|_{t=0}=-\alpha\E (fLf)=-\frac{\alpha}{2} \frac{|\partial S|}{|E|},
$$
where $L$ is the Laplacian and $|E|$ is the number of edges of the graph $G$.
\end{rmk}

\subsection{Tree}

Now we discuss the problem of noise stability in a network in terms of a tree, which gives the geometry of the problem. This was initially proposed by Mossel et al. \cite{MODRSS06} for the NICD problem.
 
We denote by $T$ an undirected tree, which gives the geometry of the problem. The edges of $T$ will be thought of as independent memoryless BSC($\epsilon$) channels with the cross-over probability $\epsilon\in[0, 1/2]$. Let $V$ denote the vertices of $T$. We refer to $S\subset V$ as the locations of the players. Some vertex $u$ of $T$ broadcasts a uniformly random string $X^u\in\{0, 1\}^n$. This string follows
the BSC($\epsilon$) edges of $T$  and eventually reaches all vertices. It is easy to see that the choice $u$ does not matter, in the
sense that the resulting joint probability distribution on strings for all vertices
is the same regardless of $u$.
Upon receiving their
strings $Y^v\in\{0, 1\}^n, v\in S$, each player applies a balanced Boolean function $f_v: \{0, 1\}^n\mapsto \{0, 1\}$, producing one output
bit. As usual, the goal of the players is to maximize 
$$
\E\prod_{v\in S}f_v(Y^v)=\P(f_v(Y^v)=1, v\in S)
$$ 
without
any further communication. Note that the problem of $\alpha$-stability with $\alpha=k$
studied in Section \ref{sec:main} is just this generalized noise stability on the star graph of $k+1$ vertices with the players
at the $k$ leaves.

 In the case of NICD on the path graph, Mossel et al. \cite{MODRSS06} proved  that the best strategy for all players is to use the same dictator function (see Theorem 5.1). In the general case, they showed that there always exists an optimal protocol in which all players use monotone functions (see Theorem 6.3). A careful check of their proofs shows that their arguments also yield the following analogues on the problem of noise stability. Hence, we omit the proofs.

\begin{thm}
Suppose that $T$ is a path of length $k$ on the set $\{0, 1, \cdots, k\}$. Let $S=\{i_0, \cdots, i_l\}$ be a subset of size at least two. Then we have
$$
\E\prod_{v\in S}f_v(Y^v)\leq 2^{-(l+1)}\prod_{j=1}^l(1+(1-2\epsilon)^{i_j-i_{j-1}}).
$$
Equality is achieved if and only if $f_v$ are the identical dictator function.
\end{thm}

\begin{thm}
For any tree $T$, the maximal correlation $\E\prod_{v\in S}f_v(Y^v)$ can be achieved by some monotone Boolean function.
\end{thm}

\section{Discussion}\label{sec:conclusion}
We investigate the noise stability of Boolean functions in various settings, such as functions defined on discrete cube, discrete torus, as well as in a tree model. Characterizations of extremal functions are given in different scenarios. Close connections with the problem of NICD and the conjecture of Courtade-Kumar on the most informative Boolean function are discussed. This paper significantly generalizes our earlier work \cite{LM18} with the focus on the discrete cube case. Regarding practical applications, our study of the discrete torus model is potentially useful for communications via low-noise channels with phase-shift keying (PSK) modulation. For example, our study of the discrete torus model captures the character of the $\ell$-PSK schemes with errors limited to a phase shift of $2\pi/\ell$ or $-2\pi/\ell$, say each with probability $\epsilon/2$, i.e., the errors remain closest to the original signal. Future work may consider general non-negative functions on discrete cube and Boolean functions on general product measure spaces. Analogous questions can be asked for general Markov semi-groups. Extension of the tree model in Section \ref{sec:extensions} to networks of general graphs is  interesting from both theoretical and practical perspectives. It might be worth to explore the connection between this $\alpha$-stability problem and Talagrand's convolution conjecture \cite{Tal89}.

\section*{Acknowledgment}

We would like to thank Alex Samorodnitsky for pointing out the reference \cite{ABCJN17}. We are indebted to Imre Leader for his clarification of results in \cite{BL91}. We thank Nathan Keller for pointing out that the noise operator preserves the monotonicity (Lemma \ref{lem:max-1}). We also appreciate the anonymous referee for pointing out several inaccuracies and many valuable comments. This work is supported by NSF grant CCF-1527270.

\end{document}